\documentclass[11pt,oneside,english]{amsart}
\usepackage[T1]{fontenc}
\usepackage[latin9]{inputenc}
\usepackage{geometry}
\usepackage{textcomp}
\usepackage{amstext}
\usepackage{amsthm}

\makeatletter
\numberwithin{equation}{section}
\numberwithin{figure}{section}
\theoremstyle{plain}
\newtheorem{thm}{\protect\theoremname}
  \theoremstyle{plain}
  \newtheorem{lem}[thm]{\protect\lemmaname}
  \theoremstyle{remark}
  \newtheorem{rem}[thm]{\protect\remarkname}
  \theoremstyle{plain}
  \newtheorem{prop}[thm]{\protect\propositionname}

\def\makebbb#1{
    \expandafter\gdef\csname#1\endcsname{
        \ensuremath{\Bbb{#1}}}
}\makebbb{R}\makebbb{N}\makebbb{Z}\makebbb{C}\makebbb{H}\makebbb{E}\makebbb{H}\makebbb{P}\makebbb{B}\makebbb{Q}\makebbb{E}

\usepackage{babel}

\usepackage{babel}

\makeatother

\usepackage{babel}

\makeatother

\usepackage{babel}
  \providecommand{\lemmaname}{Lemma}
  \providecommand{\propositionname}{Proposition}
  \providecommand{\remarkname}{Remark}
\providecommand{\theoremname}{Theorem}

\begin{document}

\title{On the strict convexity of the K-energy}

\author{Robert J. Berman}
\begin{abstract}
Let $(X,L)$ be a polarized projective complex manifold. We show,
by a simple toric one-dimensional example, that Mabuchi's K\textendash energy
functional on the geodesically complete space of bounded positive
$(1,1)-$forms in $c_{1}(L),$ endowed with the Mabuchi-Donaldson-Semmes
metric, is not strictly convex modulo automorphisms. However, under
some further assumptions the strict convexity in question does hold
in the toric case. This leads to a uniqueness result saying that a
finite energy minimizer of the K-energy (which exists on any toric
polarized manifold $(X,L)$ which is uniformly K-stable) is uniquely
determined modulo automorphisms under the assumption that there exists
some minimizer with strictly positive curvature current.
\end{abstract}

\email{robertb@chalmers.se}
\maketitle

\section{Introduction}

Let $(X,L)$ be a polarized compact complex manifold and denote by
$\mathcal{H}$ the space of all smooth metrics $\phi$ on the line
bundle $L$ with strictly positive curvature, i.e the curvature two-form
$\omega_{\phi}$ of $\phi$ defines a Kähler metric on $X.$ A leading
role in Kähler geometry is played by Mabuchi's K-energy functional
$\mathcal{M}$ on $\mathcal{H},$ which has the property that a metric
$\phi$ in $\mathcal{H}$ is a critical point for $\mathcal{M}$ if
and only if the Kähler metric $\omega_{\phi}$ has constant scalar
curvature \cite{mab0}. From the point of view of Geometric Invariant
Theory (GIT) the K-energy can, as shown by Donaldson \cite{do00},
be identified with the Kempf-Ness ``norm-functional'' for the natural
action of the group of Hamiltonian diffeomorphisms on the space of
all complex structures on $X,$ compatible with a given symplectic
form. 

As shown by Mabuchi the functional $\mathcal{M}$ is convex along
geodesics in the space $\mathcal{H}$ endowed with its canonical Riemannian
metric (the Mabuchi-Semmes-Donaldson metric). More precisely, as indicated
by the GIT interpretation, $\mathcal{M}$ is \emph{strictly} convex
modulo the action of the automorphism group $G$ of $(X,L)$ in the
following sense: let $\phi_{t}$ be a geodesic in $\mathcal{H}$ (parametrized
so that $t\in[0,1])$ then
\begin{equation}
t\mapsto\mathcal{M}(\phi_{t})\,\text{is affine}\,\text{iff\,\ensuremath{\phi}(t)=}g(t)\phi_{0},\label{eq:M affine iff intro}
\end{equation}
 where $g(t)\phi_{0}$ denotes the action on $\phi_{0}$ of a one-parameter
subgroup in $G,$ i.e. $\phi_{t}$ is equal to the pull-back of $\phi_{0}$
under the flow of a holomorphic vector field on $X.$ However, the
boundary value problem for the geodesic equation in $\mathcal{H}$
does not, in general, admit strong solutions \cite{l-v}. In order
to bypass this complication Chen introduced a natural extension of
$\mathcal{M}$ to the larger space $\mathcal{H}_{1,1}$ consisting
of all (singular) metrics $\phi$ on $L$ such that the curvature
$\omega_{\phi}$ is defined as an $L^{\infty}-$form. The advantage
of the latter space is that it is geodesically convex (in the sense
of metric spaces; see \cite{da} and references therein). It was conjectured
by Chen \cite{c1} and confirmed in \cite{b-b2} that $\mathcal{M}$
is convex on $\mathcal{H}_{1,1}.$ However, the question whether $\mathcal{M}$
is strictly convex modulo the action of $G$ on $\mathcal{H}_{1,1}$
was left open in \cite{b-b2}. In this short note we give a negative
answer to the question already in the simplest case when $X$ is the
Riemann sphere and the metrics are $S^{1}-$invariant.
\begin{thm}
\label{thm:counterex}Let $L\rightarrow X$ be the hyperplane line
bundle over the Riemann sphere. There exists an $S^{1}-$invariant
geodesic $\phi_{t}$ in $\mathcal{H}_{1,1}$ such that $\mathcal{M}(\phi_{t})$
is affine, but $\phi_{t}$ is not of the form $g(t)\phi_{0}.$ 
\end{thm}

In terms of the standard holomorphic coordinate $z$ on the affine
piece $\C$ of $X$ the geodesic $\phi_{t}$ in the previous theorem
can be taken so that $\phi_{t}$ is equal to the Fubini-Study metric
$\phi_{0}$ on the lower-hemisphere, flat on a $t-$dependent collar
attached to the equator and then glued to a one-parameter curve of
the form $g(t)^{*}\phi_{0}$ in the remaining region of the upper
hemi-sphere. 

The failure of the strict convexity of $\mathcal{M}$ on $\mathcal{H}_{1,1}$
modulo $G$ appears to be quite surprising in view of the fact that
the other canonical functional in Kähler geometry - the Ding functional
$\mathcal{D}$ - is strictly convex on $\mathcal{H}_{1,1}$ modulo
$G.$ In fact, the Ding functional (which is only defined in the ``Fano
case'', i.e. when $L$ is the anti-canonical line bundle over a Fano
manifold) is even strictly convex modulo the action of $G$ on the
space of all $L^{\infty}-$metrics on $L$ with positive curvature
current \cite{bern} . 

One important motivation for studying strict convexity properties
of the Mabuchi functional \emph{$\mathcal{M}$ }on suitable completions
of $\mathcal{H}$ comes from the \emph{Yau-Tian-Donaldson conjecture.
}In its uniform version the conjecture says that the first Chern class
$c_{1}(L)$ of $L$ contains a Kähler metric of constant scalar curvature
iff $(X,L)$ is uniformly K-stable (in the $L^{1}-$sense) relative
to a maximal torus of $G$ \cite{bjh1,der}. The ``only if'' direction
was established in \cite{bdl2} when $G$ is trivial (and a similar
proof applies in the general case \cite{bbhj}). The proof in \cite{bdl2}
uses the convexity of $\mathcal{M}$ on the finite energy completion
$\mathcal{E}^{1}$ of $\mathcal{H}.$\footnote{The space $\mathcal{E}^{1}$ was originally introduced in \cite{g-z1}
from a pluripotential point of view and, as shown in \cite{da}, $\mathcal{E}^{1}$
may be identified with the metric completion of $\mathcal{H}$ with
respect to the $L^{1}-$Finsler version of the Mabuchi-Semmes-Donaldson
metric on $\mathcal{H}$} The remaining implication in the Yay-Tian-Donaldson conjecture is
still widely open, in general, but a first step would be to establish
the existence of a minimizer of \emph{$\mathcal{M}$} in the finite
energy space $\mathcal{E}^{1}$, by generalizing the variational approach
to the ``Fano case'' introduced in \cite{bbj}. Leaving aside the
challenging question of the regularity of a minimizer one can still
ask if the minimizer is canonical, i.e. uniquely determined modulo
$G?$ (as conjectured in \cite{d-r}). The uniqueness in question
would follow from the strict convexity of $\mathcal{M}$ on $\mathcal{E}^{1}$
modulo the action of $G.$ However, by the previous theorem such a
strict convexity does not hold, in general. On the other hand, it
would be enough to establish the following weaker strict convexity
property: 
\[
t\mapsto\mathcal{M}(\phi_{t})\,\text{is constant}\,\implies\,\ensuremath{\phi}(t)=g(t)\phi_{0},
\]
(the converse implication holds if $(X,L)$ is K-stable). This may
still be too optimistic, but here we observe that this approach towards
the uniqueness problem can be made to work in the toric case if one
assumes some a priori positivity of the curvature current of some
minimizer.
\begin{thm}
\label{thm:toric}Let $(X,L)$ be an $n-$dimensional polarized toric
manifold. Assume that $(X,L)$ is uniformly K-stable relative the
torus action. Then there exists a finite energy minimizer $\phi$
of $\mathcal{M}$ and the minimizer is unique modulo the action of
$\C^{*n}$ under the assumption that there exists some finite energy
minimizer $\phi_{0}$ whose curvature current is strictly positive
on compacts of the dense open orbit of $\C^{*n}$ in $X.$ 
\end{thm}

\subsection{Relations to previous results}

In view of its simplicity it is somewhat surprising that the counterexample
in Theorem \ref{thm:counterex} does not seem to have been noticed
before. The key point of the proof is a generalization of Donaldson
formula for the Mabuchi functional $\mathcal{M}$ in the smooth toric
setting to a singular setting (see Lemma \ref{lem:D formula in one dim sing}
and Lemma \ref{lem:sing dona in toric}), showing that 
\begin{equation}
\mathcal{M}(\phi)=\mathcal{F}(u),\label{eq:M is F intro}
\end{equation}
where the non-linear part of the functional $\mathcal{F}$ only depends
on the non-singular part (in the sense of Alexandrov) of the Hessian
of the convex function on the moment polytope of $(X,L),$ corresponding
to the metric $\phi.$ This leads, in fact, to a whole class of counter-examples
to the strict convexity in question, by taking $\phi_{t}$ to be any
torus invariant geodesic $\phi_{t},$ emanating from a given $\phi_{0}\in\mathcal{H},$
whose Legendre transform $u_{t}$ is of the form 
\[
u_{t}=u_{0}+tv,
\]
for a convex and piece-wise affine function $v.$ Such a curve defines
a geodesic ray associated to a toric test configuration for $(X,L)$
and the fact that $\phi_{t}\in\mathcal{H}_{1,1}$ then follows from
general results (see \cite[Section 7]{c-t,p-s,s-z}). A by-product
of formula \ref{eq:M is F intro} is a slope formula for the K-energy
along toric geodesics of finite energy (see Section \ref{lem:sing dona in toric}).

The functional $\mathcal{F}$ has previously appeared in a series
of papers by Zhou-Zhu (see Remark \ref{rem:zz}). In particular, it
was shown in \cite[Section 6]{z} (when $n=2)$ that a minimizer $u$
of $\mathcal{F}$ is uniquely determined, modulo the complexified
torus action, under the stronger assumption that there exists some
minimizer $u_{0}$ of $\mathcal{F}$ which is $C^{\infty}-$smooth
and strictly convex in the interior (while our assumption is equivalent
to merely assuming that $u$ is $C^{1,1}-$smooth in the interior).
We also recall that in the toric surface case (i.e when $n=2)$ it
was shown by Donaldson \cite{d0} that uniform K-stability is equivalent
to K-stability. Moreover, the Yau-Tian-Donaldson in the latter case
was settled by Donaldson in a series of papers culminating in \cite{do}
(as a consequence, any minimizer of $\mathcal{M}$ is then smooth
and positively curved). Theorem \ref{thm:toric} should also be compared
with the general uniqueness result for finite energy minimizers of
$\mathcal{M}$ on any Kähler manifold which holds under the assumption
that there exists some minimizer which is smooth and strictly positively
curved. The proof of the latter result, which generalizes the uniqueness
result in \cite{b-b2}, exploits a weaker form of strict convexity
which holds on the linearized level around a bona fide metric with
constant scalar curvature.

\subsubsection*{Acknowledgment}

This work was supported by grants from the ERC and the KAW foundation.
I am grateful to Long Li for comments on the first version of the
paper.

\section{Proofs}

We start with some preparations. To keep things as simple as possible
we mainly stick to the one-dimensional situation (see \cite{b-b}
for the general convex analytical setup and its relations to polarized
toric varieties).

\subsection{Convex preparations}

Let $\phi(x)$ be a lower semi-continuous (lsc) convex function on
$\R$ (taking values in $]-\infty,\infty]).$ Its point-wise derivative
$\phi'(x)$ exists a.e. on $\R$ and defines an element in $L_{loc}^{\infty}(\R).$
We will denote by $\partial\phi$ the subgradient of $\phi,$ which
is a set-valued map on $\R$ with the property that $(\partial\phi)(x)$
is a singleton iff $\phi'(x)$ exists at $x.$ Similarly, we will
denote by $\partial^{2}\phi$ the measure on $\R$ defined by the
second order distributional derivative of $\phi.$ By Lebesgue's theorem
we can decompose 
\[
\partial^{2}\phi=\partial_{s}^{2}\phi+\phi'',
\]
 where $\partial_{s}^{2}\phi$ denotes the singular part of the measure
$\partial^{2}\phi$ and $\phi''\in L_{loc}^{1}(\R)$ denotes the regular
part (wrt Lebesgue measure $dx$), which coincides with the second
order derivative of $\phi$ almost everywhere on $x.$ We set 
\begin{equation}
\phi_{0}(x):=\log(1+e^{x})\label{eq:def of phi not}
\end{equation}
 and 
\[
\mathcal{P}_{+}(\R):=\left\{ \phi:\,\,\text{\ensuremath{\phi}\,\ensuremath{\text{convex\,on\,\ensuremath{\R,\,\,\phi-\phi_{0}\in L^{\infty}(\R)}}}}\right\} 
\]
( $\partial^{2}\phi$ is a probability measure for any such $\phi$).
Given a function $\phi$ in $\mathcal{P}_{+}(\R)$ we will denote
by $u$ its \emph{Legendre transform} which defines a finite lsc convex
function $u$ on $[0,1]$ (which is equal to $\infty$ on $]0,1[^{c})$
\[
u(y):=(\phi^{*})(y):=\sup_{x\in\R}xy-\phi(x)
\]
Since $\phi=u^{*}$ the map $\phi\mapsto u$ gives a bijection
\[
\mathcal{P}_{+}(\R)\longleftrightarrow\left\{ u:\,\,\text{\ensuremath{u}\,\ensuremath{\text{convex\,on\,\ensuremath{[0,1]}}}}\right\} \cap L^{\infty}[0,1]
\]
 which is an isometry wrt the $L^{\infty}-$norms. Moreover, $\phi$
is smooth and strictly convex on $\R$ iff $u$ is smooth and strictly
convex on $]0,1[,$ as follows from the formula 
\begin{equation}
\phi(x)=xy-u'(y),\label{eq:legendre form}
\end{equation}
 if $x=u'(y)$ and $u$ is differentiable at $y$ (and vice versa
if $\phi$ is replaced by $u).$ Moreover, if $u$ is two times differentiable
at $y=0$ and $u''(y)>0$ then $\phi''$ is differentiable at $x$
and
\begin{equation}
\phi''(x)=1/u''(y)\label{eq:second deriv as inverse}
\end{equation}

\subsection{Complex preparations}

Let $(X,L):=(\P^{1},\mathcal{O}(1))$ be the complex projective line
$\P^{1}$ endowed with the hyperplane line bundle $\mathcal{O}(1).$
Realizing $\P^{1}$ as the Riemann sphere (i.e. the one-point compactification
of the complex line $\C)$ a locally bounded metric $\Phi$ on $\mathcal{O}(1)$
may, in the standard way, be identified with a convex function $\Phi(z)$
on $\C$ such that 
\[
\Phi-\Phi_{0}\in L^{\infty}(\C),\,\,\,\,\Phi_{0}(z):=\log(1+|z|^{2}),
\]
 where $\Phi_{0}$ corresponds to the Fubini-Study metric on $\mathcal{O}(1),$
which defines a smooth metric with strictly positively curvature $\omega_{0}$
on $\P^{1}$(coinciding with the standard $SU(2)-$ invariant two-form
on $\P^{1}).$ Moreover, the metric $\Phi$ on $\mathcal{O}(1)$ has
semi-positive curvature $\omega_{\Phi}$ on $\P^{1}$ iff $\Phi(z)$
is subharmonic on $\C.$ More precisely, 
\begin{equation}
\omega_{\Phi|\C}=\frac{i}{2\pi}\partial\bar{\partial}\Phi:=\frac{i}{2\pi}\frac{\partial^{2}\Phi}{\partial z\partial\bar{z}}dz\wedge\bar{dz}\label{eq:expr for curvature}
\end{equation}
We will denote by $\mathcal{H}_{b}^{S^{1}}$ the space of all bounded
(i.e $L^{\infty})$ metrics $\Phi$ on $\mathcal{O}(1)\rightarrow\P^{1}$
with semi-positive curvature which are $S^{1}-$invariant (wrt the
standard action of $S^{1}).$ Setting 
\[
x:=\log|z|^{2}
\]
gives a correspondence 
\[
\mathcal{H}_{b}^{S^{1}}\longleftrightarrow\mathcal{P}_{+}(\R),\,\,\,\Phi(z)\mapsto\phi(x):=\Phi(e^{2x})
\]
As is well-known, under the Legendre transform the Mabuchi-Donaldson-Semmes
metric on $\mathcal{H}^{T}$ corresponds the standard flat metric
induced from $L^{2}[0,1].$ It follows that a geodesic $\Phi_{t}$
in $\mathcal{H}_{b}^{S^{1}}$ corresponds to a curve $\phi_{t}$ in
$\mathcal{P}_{+}(\R)$ with the property that the corresponding curve
\[
u_{t}:=\phi_{t}^{*}
\]
of bounded convex functions on $[0,1]$ is affine wrt $t.$

In particular, taking $\phi_{0}$ to be defined by \ref{eq:def of phi not},
the following curve defines a geodesic in $\mathcal{H}_{b}^{S^{1}}:$ 

\[
\phi_{t}(x):=(u_{0}+tv)^{*},\,\,\,\,u_{0}(y):=\phi_{0}^{*}(y)=y\log y+(1-y)\log(1-y),
\]
 where $v$ is the following convex piece-wise affine function on
$[0,1]:$

\[
v(y):=0,\,\,\,y\in[0,\frac{1}{2}],\,\,\,v(y)=-\frac{1}{2}+y,\,\,y\in[\frac{1}{2},1]
\]
\begin{lem}
\label{lem:phi t}Assume that $t\in]0,1].$ Then $\phi_{t}$ defines
an $S^{1}-$invariant metric on $\mathcal{O}(1)\rightarrow\P^{1}$
which is $C^{1,1}-$smooth on $\P^{1}$ with semi-positive curvature
and $C^{\infty}-$ smooth and strictly positively curved on the complement
of a $t-$dependent neighborhood of the equator. More precisely, in
the logarithmic coordinate $x\in\R$ we have $\phi_{t}(x)=\phi_{0}(x)$
when $x\leq0$ and $\phi_{t}(x)=\phi_{0}(0)+x/2$ when $x\in[0,t]$
and $\phi_{t}(x)=\phi_{0}(x-t)+t/2$ when $x\geq t.$
\end{lem}

\begin{proof}
\emph{Step 1:} $\phi_{t}$ is in $C^{1,1}(\R).$ 

This step only uses that $v$ is convex and bounded on $[0,1].$ Let
$v^{(j)}$ be a sequence of smooth strictly convex function on $[0,1]$
such that $\left\Vert v^{(j)}-v\right\Vert _{L^{\infty}[0,1]}\rightarrow0.$
Set $u^{(j)}(t):=u_{0}+tv^{(j)}$ and $\phi^{(j)}:=u^{(j)}{}^{*}.$
Since the Legendre transform preserves the $L^{\infty}-$norm we have
$\left\Vert \phi^{(j)}-\phi\right\Vert _{L^{\infty}(\R)}\rightarrow0.$
By construction, $u_{t}^{(j)}$ is smooth and satisfies $(u_{t}^{(j)})^{''}\geq(u_{0})^{''}\geq1/C>0.$
Hence, by \ref{eq:second deriv as inverse}, $(\phi_{t}^{(j)})^{''}\leq C$
and letting $j\rightarrow\infty$ thus implies that $\partial^{2}\phi=\phi''\leq C,$
showing that $\phi_{t}$ is in $C^{1,1}(\R).$ 

\emph{Step 2: }Explicit description of $\phi_{t}$

We fix $t>0$ and observe that the map $y\mapsto u_{t}'(y)$ induces
two diffeomorphisms (with inverses $x\mapsto\phi_{t}'(x))$
\begin{equation}
u_{t}':\,\,\,Y_{-}:=]0,\frac{1}{2}[\rightarrow X_{-}:=]-\infty,0[,\,\,\,Y_{+}:=]\frac{1}{2},1[\rightarrow X_{-}:=]t,\infty[\label{eq:diffeom}
\end{equation}
This follows directly from the fact that $u_{t}^{'}$ is strictly
positive on $Y_{\pm}$ and converges to $0$ and $t$ when $y\rightarrow1/2$
from the left and right, respectively. We claim that this implies,
by general principles, that the restriction of $\phi_{t}$ to $X_{\pm}$
only depends on the restriction of $u_{t}$ to $Y_{\pm}.$ Indeed,
if $x=u_{t}'(y)$ for $y\in Y_{\pm}$ then formula \ref{eq:legendre form}
shows that $\phi_{t|X_{\pm}}$ only depends on the restriction of
$u$ to $X_{-}.$ Hence, the restriction of $\phi_{t}$ to $X_{-}$
is given by $u_{0}^{*}=\phi_{0}$ and the restriction of $\phi_{t}$
to $X_{+}$ is given by the Legendre transform of $u_{0}(y)+t(y-1/2),$
which is equal to $\phi_{0}(x-t)+t/2.$ Finally, it follows from the
diffeomorphism \ref{eq:diffeom} (using, for example, that $\phi_{t}'$
is increasing) that $\phi_{t}'=1/2$ on $[0,t].$ 
\end{proof}
\begin{rem}
A more symmetric form of the geodesic $\phi_{t}$ may be obtained
by setting $\tilde{\phi}_{t}(x):=2\phi_{t}(x)-x,$ which has the property
that $\tilde{\phi}_{t}(x):=\tilde{\phi}_{0}(x):=\log(e^{-x}+e^{x})$
when $x\leq0$ and $\tilde{\phi}_{t}(x)=\tilde{\phi}_{0}(0)$ when
$x\in[0,t]$ and $\tilde{\phi}_{t}(x)=\tilde{\phi}_{0}(x-t)$ when
$x\geq t.$ Geometrically, $\tilde{\phi}_{t}$ defines a geodesic
ray of metrics on $\mathcal{O}(2)\rightarrow\P^{1},$ expressed in
terms of the trivialization of $\mathcal{O}(2)$ over $\C^{*}\subset\P^{1}$
induced from the embedding $\C^{*}\rightarrow\C^{2}\rightarrow\P^{2}$
defined by $F(z):=(z^{-1},z)\in\C^{2},$ where $X$ is identified
with the closure $\overline{F(\C)}$ of $F(\C)$ in $\P^{2}$ and
$\mathcal{O}(2)$ with the restriction of $\mathcal{O}_{\P^{2}}(1)$
to $\overline{F(\C)}.$ A direct calculation reveals that $\tilde{\phi}_{t}(x)$
(and hence also $\phi_{t}(x))$ is in fact $C^{1,1}-$smooth when
viewed as a function on $\R\times\R.$ This implies that the Laplacian
of the corresponding local potentials over $\P^{1}\times D^{*}$ (where
$D^{*}$ denotes the punctured unit disc with holomorphic coordinate
$\tau$ such that $t:=-\log|\tau|^{2}$) is locally bounded, i.e.
the geodesic has Chen's regularity \cite{c0} in the space-time variables.
It should also be pointed out that $\tilde{\phi}_{t}$ can be realized
as the geodesic ray, emanating from the Fubini-Study metric, associated
to the toric test configuration of $(X,L):=(\P^{1},\mathcal{O}(2))$
determined (in the sense of \cite{d0,c-t,s-z}) by the piece-wise
affine function $\tilde{v}(y)=\max\{0,y\}$ on the moment polytope
$[-1,1]$ of $(\P^{1},\mathcal{O}(2))$ (as in \cite{c-t}). Using
this realization the $C^{1,1}-$ regularity also follows from the
general results in \cite{p-s,c-t} which show that the Laplacian (or
equivalently complex Hessian) of the corresponding potential is locally
bounded over $X\times D^{*}.$ Indeed, in the toric setting boundedness
of the complex Hessian is equivalent to boundedness of the real Hessian,
i.e. to $C^{1,1}-$regularity.
\end{rem}

\subsubsection{The K-energy}

Let $(X,L)$ be a polarized compact complex manifold. We recall that
the K-energy functional was originally defined by Mabuchi \cite{mab0}
on the space $\mathcal{H}$ of all smooth metrics $\Phi$ on $L$
with strictly positive curvature by specifying its differential (more
precisely, this determines $\mathcal{M}$ up to an additive constant).
Chen extended $\mathcal{M}$ to the space $\mathcal{H}_{1,1}$ consisting
of all (singular) metrics $\phi$ on $L$ such that the curvature
$\omega_{\Phi}$ of $\Phi$ is defined as an $L^{\infty}-$form \cite{c1}.
The extension is based on the Chen-Tian formula for $\mathcal{M}$
on $\mathcal{H}$ which may be expressed as follows in terms of a
fixed Kähler form $\omega_{0}$ on $X:$ 
\begin{equation}
\mathcal{M}(u)=\left(\frac{\bar{R}}{n+1}\mathcal{E}(\Phi)-\mathcal{E}^{\mbox{Ric}\ensuremath{\omega_{0}}}(\Phi)\right)+H_{\omega_{0}^{n}}(\omega_{\Phi}^{n}),\,\,\,\,\,\bar{R}:=\frac{nc_{1}(X)\cdot[\omega_{0}]^{n-1}}{[\omega_{0}]^{n}},\label{eq:formula for mab}
\end{equation}
where
\begin{equation}
H_{\mu_{0}}(\mu):=\int_{X}\log\left(\frac{\mu}{\mu_{0}}\right)\mu\label{eq:def of entr in smooth case}
\end{equation}
and$\mathcal{E}$ and $\mathcal{E}^{\mbox{Ric}\ensuremath{\omega_{0}}}$
are defined, up to an additive constant, by their differentials on
$\mathcal{H}:$ 
\begin{equation}
d\mathcal{E}_{|\Phi}=(n+1)\omega_{\Phi}^{n},\,\,\,d\mathcal{E}_{|\Phi}^{\mbox{Ric}\ensuremath{\omega_{0}}}=n\omega_{\Phi}^{n-1}\wedge\mbox{Ric}\ensuremath{\omega_{0}}\label{eq:differentials of energy and T-en}
\end{equation}
with $\mbox{Ric}\ensuremath{\omega_{0}}$ denoting the two-form defined
by the Ricci curvature of $\omega_{0}$ (see \cite{b-b2} for a simple
direct proof of the Chen-Tian formula). The extension of $\mathcal{M}$
to $\mathcal{H}_{1,1}$ is obtained by observing that both terms appearing
in the rhs of formula \ref{eq:formula for mab} are well-defined (and
finite) when $\Phi\in\mathcal{H}_{1,1}.$ We note that the functional
appearing in the first bracket of the formula is continuous wrt the
$L^{\infty}-$ norm on $\mathcal{H}_{1,1}.$ Indeed, it follows readily
from the definitions that both $\mathcal{E}$ and $\mathcal{E}^{\mbox{Ric}\ensuremath{\omega_{0}}}$
are even Lip continuous wrt the $L^{\infty}-$ norm.

In the present setting where $X=\P^{1}$ we can, for concreteness,
take $\omega_{0}=\omega_{\Phi_{0}},$ whose restriction to $\C$ is
equal to a constant times $e^{-2\Phi_{0}}dz\wedge d\bar{z}.$ 

\subsection{Conclusion of proof of Theorem \ref{thm:counterex}}

The proof will follow from the following extension to $\mathcal{H}_{1,1}^{S^{1}}$
of a formula due to Donaldson when $\Phi\in\mathcal{H}^{S^{1}}$ \cite[Prop 3.2.8]{d0}. 
\begin{lem}
\label{lem:D formula in one dim sing}Assume that $\Phi$ is in $\mathcal{H}_{1,1}^{S^{1}}.$
Then 
\[
\mathcal{M}(\Phi)=\mathcal{L}(u)-\int_{[0,1]}\log(u''(y))dy,\,\,\,\mathcal{L}(u)=\frac{1}{2}(u(1)+u(0))-\int_{0}^{1}u(y)dy
\]
where $u''\in L_{loc}^{1}$ denotes the non-singular part of $\partial^{2}u.$
\end{lem}

\begin{proof}
\emph{Step 1:} Assume that $\Phi\in\mathcal{H}_{1,1}^{S^{1}}.$ Then
\[
\mathcal{M}(\Phi)=\mathcal{L}(u)+\int_{\R}\phi^{''}(x)\log\phi^{''}(x)dx
\]

In the case when $\Phi\in\mathcal{H}^{S^{1}}$ (or more generally
when $u$ is continuous on $[0,1]$ and smooth and strictly convex
in the interior) this follows from Donaldson's formula \cite{d0}.
To extend the formula to the case when $\Phi\in\mathcal{H}_{1,1}^{S^{1}}$
first observe that 
\begin{equation}
\int\phi_{0}(x)\phi^{''}(x)dx<\infty,\label{eq:finite moments}
\end{equation}
 as follows directly from estimating $\phi''\leq C\phi_{0}''\leq Ae^{-|x|/B}$
and $\phi_{0}(x)\leq|x|+C.$ Hence, we can rewrite the Chen-Tian formula
\ref{eq:formula for mab} as 
\begin{equation}
\mathcal{M}(\Phi)=E_{0}(\Phi)+\int_{\R}\phi^{''}(x)\log\phi^{''}(x)dx,\label{eq:functional E not}
\end{equation}
 where 
\[
E_{0}(\Phi)=\left(\frac{\bar{R}}{n+1}\mathcal{E}(\Phi)-\mathcal{E}^{\mbox{Ric}\ensuremath{\omega_{0}}}(\Phi)\right)+2\int_{\R}\phi_{0}(x)\phi''(x)dx,
\]
Now take a sequence $\Phi_{j}\in\mathcal{H}^{S^{1}}$ such that $\left\Vert \Phi_{j}-\Phi\right\Vert _{L^{\infty}}\rightarrow0$
(which equivalently means that $\left\Vert u_{j}-u\right\Vert _{L^{\infty}[0,1]}\rightarrow0)$
and $\omega_{\Phi_{j}}\leq C\omega_{\Phi_{0}},$ i.e. 
\begin{equation}
\phi_{j}^{''}(x)\leq C\phi_{0}^{''}(x)\label{eq:bound on second deriv}
\end{equation}

We claim that
\begin{equation}
E_{0}(\Phi_{j})\rightarrow E_{0}(\Phi).\label{eq:conv of E not}
\end{equation}

Indeed, as pointed out above the first term appearing in the definition
of $E_{0}$ is continuous wrt the $L^{\infty}-$norm. To handle the
second term first observe that, since $\left\Vert \Phi_{j}-\Phi\right\Vert _{L^{\infty}(X)}\rightarrow0,$
the probability measures $\phi_{j}^{''}(x)dx$ converge weakly towards
$\phi^{''}(x)dx$ and hence, for any fixed $R>0,$ 
\[
\lim_{j\rightarrow\infty}\int_{|x|\leq R}\phi_{0}(x)\phi_{j}''(x)dx=\int_{|x|\leq R}\phi_{0}(x)\phi''(x)dx
\]
Moreover, the uniform bound \ref{eq:bound on second deriv} gives
\[
\limsup_{R\rightarrow\infty}\limsup_{j\rightarrow\infty}\int_{|x|\geq R}\phi_{0}\phi_{j}''(x)dx\leq C\limsup_{R\rightarrow\infty}\int_{|x|\geq R}\phi_{0}(x)\phi_{0}''(x)dx=0
\]
Hence, letting first $j\rightarrow\infty$ and then $R\rightarrow\infty$
proves \ref{eq:conv of E not}.

Now take a sequence $\Phi_{j}\in\mathcal{H}^{S^{1}}$ such that $\left\Vert \Phi_{j}-\Phi\right\Vert _{L^{\infty}}\rightarrow0$
which equivalently means that $\left\Vert u_{j}-u\right\Vert _{L^{\infty}[0,1]}\rightarrow0.$
By Donaldson's formula 
\[
E_{0}(\phi_{j})=\mathcal{L}(u_{j})
\]
and since both sides are continuous wrt the convergence of $\Phi_{j}$
towards $\Phi$ this concludes the proof of Step 1. 

\emph{Step 2: }Let $\phi$ be a convex function on $\R$ such that
$\partial^{2}\phi$ is a probability measure which is absolutely continuous
wrt $dx.$ Then 
\begin{equation}
\int_{\R}\phi^{''}(x)\log\phi^{''}(x)dx=-\int_{[0,1]}\log(u''(y))dy,\label{eq:entropy via leg}
\end{equation}
 if the left hand side is finite (and then $u''(y)>0$ a.e.).

This formula is a special case of McCann's Monotone change of variables
theorem \cite[Theorem 4.4]{mc}. But it may be illuminating to point
out that a simple direct proof can be given in the present setting
when $\phi$ is of the form $\phi_{t}$ appearing in Lemma \ref{lem:phi t}.
Indeed, then $\rho:=\phi''=0$ on a closed intervall $I$ of $\R$
and $\phi'$ diffeomorphism of the complement $I$ onto $]0,1[-\{1/2\}.$
Since $\rho\log\rho=0$ if $\rho=0$ the formula \ref{eq:entropy via leg}
then follows directly from making the change of variables $y=\phi'(x)$
on $\R-S.$ 
\end{proof}
Now, let $\Phi_{t}$ be the geodesic in $\mathcal{H}_{1,1}^{S^{1}}$
defined by the curve $\phi_{t}$ appearing in Lemma \ref{lem:phi t}.
Since $v$ is piece-wise affine we have $u_{t}''=u_{0}''$ a.e on
$\R$ and hence the previous lemma gives 
\[
\mathcal{M}(\Phi_{t})=-\int_{[0,1]}\log(u''(y))dy+t\mathcal{L}(v)
\]
 which is affine in $t.$ Moreover, $\phi_{t}$ is not induced from
the flow of a holomorphic vector field (since this would imply that
$v$ is affine on all of $[0,1]).$ This concludes the proof of Theorem
\ref{thm:counterex}. 
\begin{rem}
The functional $E_{0}$ in formula \ref{eq:functional E not} coincides
with the (attractive) Newtonian energy of the measure $\mu=\partial^{2}\phi:$
\[
E_{0}(\mu)=\frac{1}{4}\int_{\R^{2}}|x-y|\mu(x)\otimes\mu(y)
\]
and the continuiuty property of $E_{0}$ used in the in Step 1 can
be alternatively deduce from the fact that $E$ is continuous on the
space $\mathcal{P}_{1}(\R)$ of all probability measures with finite
first moment (endowed with the $L^{1}-$Wasserstein topology). This
point of view is further developed in the higher dimensional toric
setting in \cite{b0}.
\end{rem}

\subsection{Proof of Theorem \ref{thm:toric}}

In this higher dimensional setting we will be rather brief and refer
to \cite{b0} for more details. Let $(X,L)$ be an $n-$dimensional
toric manifold and denote by $P$ the corresponding moment lattice
polytope in $\R^{n}$ which contains $0$ in its interior. We will
denote by $d\sigma$ the measure on $\partial P$ induced from the
standard integer lattice in $\R^{n}$ (which is comparable with the
Lebesgue measure on $\partial P)$ \cite{d0}. The $n-$dimensional
real torus acting on $(X,L)$ will be denoted by $T.$ As above we
can then identify a $T-$invariant metric $\Phi$ on $L$ with positive
curvature current with a convex function $\phi(x)$ on $\R^{n}$ (whose
sub-gradient maps into $P)$ and, via the Legendre transform, with
a convex function $u$ on $P.$ We will denote by $\partial^{2}\phi$
the distributional Hessian of $\phi$ and by $(\nabla^{2}\phi)(x)$
the Alexandrov Hessian of $\phi$ which is defined for almost all
$x$ (on the subset where $\phi$ is finite). 

Assume that $(X,L)$ is uniformly K-stable relative to the torus $T$
(in the $L^{1}-$sense). Concretely, this means (see \cite{h}) that
there exists $\delta>0$ such that for any rational piece-wise affine
convex function $u$ on $P,$
\begin{equation}
\mathcal{L}(u):=\int_{P}udy-c\int_{\partial P}u\geq\delta\inf_{l\in(\R^{n})^{*}}\left(\int_{P}(u-l)dy-\inf_{P}(u-l)\right),\,\,\,c:=\int_{P}dy/\int_{\partial P}d\sigma,\label{eq:unif K toric}
\end{equation}
 where the inf ranges over all linear functions $l$ on $\R^{n}$
(which, geometrically, may be identified with an element of the real
part of the Lie algebra of the complex torus). We note that, by a
standard approximation argument, the inequality \ref{eq:unif K toric}
holds for the space $\mathcal{C}(P)$ of all convex functions $u$
on $P$ such that $u\in L^{1}(P)\cap L^{1}(\partial P)$ (where $u_{|\partial P}(y)$
is defined as the radial boundary limit of $u).$\footnote{In fact, if $u$ is convex on $P$ and in $L^{1}(\partial P),$ then
automatically $u\in L^{1}(P).$ } The uniform K-stability implies, by \cite{z-z0,h}, that $\mathcal{M}$
is coercive relative to $T,$ i.e. there exist $C>0$ such that the
following coercivity inequality holds on $\mathcal{H}^{T}:$ 
\[
\mathcal{M}(\Phi)\geq\inf_{g\in\C^{*n}}J(g\Phi)/C-C,
\]
 where $J$ denotes Aubin's $J-$functional. The functional $\mathcal{M}$
admits a canonical extension to the space $\mathcal{E}^{1}$ of all
(singular) metrics on $L$ with positive curvature current and finite
energy (namely, the greatest lsc extension of $\mathcal{M}$ from
$\mathcal{H}$ to $\mathcal{E}^{1},$ endowed with the strong topology
\cite{bbegz,bdl}). The coercivity of $\mathcal{M}$ combined with
the results in \cite{bbegz} (which show that $\mathcal{M}$ is lsc
wrt the weak topology on $\mathcal{E}^{1}$) implies that there exists
a $T-$invariant minimizer $\Phi_{0}$ of $\mathcal{M}$ on the space
$\mathcal{E}^{1}(X,L)$ of all (singular) metrics on $L$ with positive
curvature current and finite energy.

A generalization of the argument used in the proof of Lemma \ref{lem:D formula in one dim sing}
gives the following lemma which extends Donaldson's formula in \cite{d0}
to the finite energy setting (the proof is given in \cite{b0}):
\begin{lem}
\label{lem:sing dona in toric}Assume that $\Phi\in\mathcal{E}^{1}(X,L)^{T}$
and $\mathcal{M}(\Phi)<\infty.$ Then $u\in\mathcal{C}(P)$ and
\begin{equation}
\mathcal{M}(\Phi)=\mathcal{F}(u):=\mathcal{L}(u)-\int_{P}\log(\nabla^{2}u(y))dy,\label{eq:D formula in sing setting}
\end{equation}
 where $\nabla^{2}u$ denotes the Alexandrov Hessian of $u$ and both
terms are finite (in particular, $\nabla^{2}u(y)>0$ a.e. on $P).$
\end{lem}

\begin{rem}
\label{rem:zz}The functional $\mathcal{F}$ on $\mathcal{C}(P)$
has previously been studied in a series of papers by Zhou-Zhu (see
\cite{z-z,z-z0}). In particular it was shown in \cite{z-z} that
$\mathcal{F}$ admits a minimizer $u.$ But the point of the previous
formula is that it identifies $\mathcal{F}$ with the Mabuchi functional
on the space $\mathcal{E}^{1}(X,L)^{T}.$ As a byproduct this gives
a new proof of the existence of a minimizer $u$ of $\mathcal{F}.$ 
\end{rem}

Let now $\Phi_{0}$ and $\Phi_{1}$ be two given minimizers of $\mathcal{M}$
in $\mathcal{E}^{1}(X,L)^{T}$ and denote by $\Phi_{t}$ the corresponding
geodesic in $\mathcal{E}^{1}(X,L)^{T}$ (which corresponds to $u_{t}:=u_{0}+t(u_{1}-u_{0})$
under the Legendre transform). By the previous lemma the function
$t\mapsto\mathcal{M}(\Phi_{t})$ decomposes in two terms, where the
first term is affine in $t$ and the second one is convex. Since $\mathcal{M}(\Phi_{t})$
is constant (and in particular affine) it follows that the second
term,
\[
t\mapsto-\int_{P}\log(\det\nabla^{2}u_{t}(y))dy
\]
 is also affine. But this forces, using the arithmetic-geometric means
inequality, that
\begin{equation}
\nabla^{2}u_{1}=\nabla^{2}u_{0}\,\,\text{a.e.\,on \ensuremath{P}}.\label{eq:Hessian same ae}
\end{equation}
As a consequence the previous function in $t$ is, in fact, constant.
Since $\mathcal{M}(\Phi_{t})$ is also constant in $t$ formula \ref{eq:D formula in sing setting}
forces $\mathcal{L}(u_{t})=\mathcal{L}(u_{0})$ for all $t.$ Setting
$v:=u_{1}-u_{0}$ this means that 
\[
\mathcal{L}(v)=0.
\]
Now, if $v$ is convex, then it follows form the assumption of uniform
relative K-stability that $v$ is affine and hence $\Phi_{0}$ and
$\Phi_{1}$ coincide modulo the action of $\C^{*n}.$ All that remains
is thus to show that $v$ is convex. To this end we invoke the assumption
that the distributional Hessian of $\phi_{0}$ satisfies 
\[
\nabla^{2}\phi_{0}\geq C_{K}I
\]
on any given compact subset $K$ of $\R^{n}.$ We claim that this
implies that $u_{0}\in C_{loc}^{1,1}(P).$ Indeed, since $\Phi_{0}$
has finite energy it has full Monge-Ampère mass and hence the closure
of the sub-gradient image $(\partial\Phi_{0})(\R^{n})$ is equal to
$P.$ It follows (just as in the proof of Lemma \ref{lem:phi t})
that 
\[
\partial^{2}u_{0}=\nabla^{2}u_{0}\leq C_{K}^{-1}I
\]
 on the closure of $(\partial\Phi_{0})(K)$ in $P.$ Since $K$ was
an arbitrary compact subset of $\R^{n}$ it follows that $u_{0}\in C_{loc}^{1,1}(P).$
The proof of the theorem is now concluded by invoking the following
lemma (see \cite[Lemma 3.2]{mc}):
\begin{lem}
Let $u_{0}$ and $u_{1}$ be two finite convex functions on an open
convex set $P\subset\R^{n}$ such that $u_{0}\in C_{loc}^{1,1}(P)$
and the Alexandrov Hessians satisfy \ref{eq:Hessian same ae}. Then
$u_{1}-u_{0}$ is convex.
\end{lem}

\subsection{\label{subsec:A-generalized-slope}A generalized slope formula for
the K-energy}

We conclude the paper by observing that a by-product of Lemma \ref{lem:sing dona in toric}
is the following generalization of the slope formula for the K-energy
in \cite{bhj2} (which concerns the case when $\Phi_{t}$ is defined
by a bona fide metric on a test configuration) to the present singular
setting:
\begin{prop}
\label{prop:(Slope-formula)-Let}(Slope formula) Let $\Phi_{t}$ be
a geodesic ray in $\mathcal{E}^{1}(X,L)^{T}$ such that $\Phi_{0}\in\mathcal{H}(X,L)^{T}$
and $\mathcal{M}(\Phi_{t})<\infty$ for any $t\in[0,\infty[.$ Then
\[
\lim_{t\rightarrow\infty}t^{-1}\mathcal{M}(\Phi_{t})=\mathcal{L}(v)<\infty
\]
 where $u_{t}=u+tv$ is the curve of convex functions in $L^{1}(\partial P)$
corresponding to $\Phi_{t}$ under Legendre transformation.
\end{prop}

\begin{proof}
Since $\mathcal{M}(\Phi_{t})<\infty$ Lemma \ref{lem:sing dona in toric}
shows that $u_{t}=u_{0}+tv\in L^{1}(\partial P)$ for all $t\geq0,$
where $v:=u_{1}-u_{0}.$ Moreover, since $u_{t}$ is convex for any
$t\geq0$ it also follows that $v$ is convex and $v\in L^{1}(\partial P).$
Now, since $\partial^{2}u_{0}$ is invertible we can, denoting the
inverse by $A(y),$ write 
\[
\int\log(\det\nabla^{2}(u_{0}+tv)(y))dy=C_{0}+\int_{P}\log(\det(1+tA(y)\nabla^{2}v_{0}(y))dy,
\]
 which is finite for any $t$ (by Lemma \ref{lem:sing dona in toric}).
Moreover, since $\nabla^{2}v_{0}(y)\geq0$ we have, when $t\geq1,$
that 
\[
0\leq\int_{P}\log(\det(I+tA(y)\nabla^{2}v_{0}(y))dy\leq\text{Vol\ensuremath{(P)}}n\log t+\int_{P}\log(\det(I+A(y)\nabla^{2}v_{0}(y))dy,
\]
where all terms are finite. Hence, dividing by $t$ and letting $t\rightarrow\infty$
concludes the proof of the proposition. 
\end{proof}
In the terminology of \cite{bjh1,bhj2,bbj} this formula shows that
the slope of the Mabuchi functional along a finite energy geodesic
is equal to the Non-Archimedean Mabuchi functional of the corresponding
(singular) Non-Archimedean metric. It would be very interesting to
extend this slope formula to the non-toric setting. Indeed, this is
the key missing ingredient when trying to extend the variational approach
to the (uniform) Yau-Tian-Donaldson conjecture in the ``Fano case''
in \cite{bbj} to a general polarized manifold $(X,L),$ in order
to produce a finite energy minimizer of $\mathcal{M}.$

\end{document}